\documentclass{tac}
\usepackage{cancel}
\usepackage{pifont} 
\usepackage[stretch=10]{microtype}

\usepackage{xy}
\usepackage{graphicx,amsfonts,latexsym,amssymb,cancel,pstricks,array}
\usepackage{amssymb,amscd,amsmath,latexsym}
\usepackage{hyperref}
\usepackage{slashbox}
\usepackage[mathscr]{euscript}
\usepackage{tikz}
\tikzset{node distance=2cm, auto}
\usetikzlibrary{matrix}
\newcommand{\scr}{\mathscr}
\newcommand{\M}{\scr{M}}
\newcommand{\MO}{\scr{M}_1}
\newcommand{\MT}{\scr{M}_2}
\newcommand{\FO}{\scr{F}_1}
\newcommand{\FT}{\scr{F}_2}
\newcommand{\WO}{\scr{W}_1}
\newcommand{\CO}{\scr{C}_1}
\newcommand{\WT}{\scr{W}_2}
\newcommand{\CT}{\scr{C}_2}
\newcommand{\IO}{\scr{I}_1}
\newcommand{\JO}{\scr{J}_1}
\newcommand{\JT}{\scr{J}_2}
\newcommand{\IT}{\scr{I}_2}
\newcommand{\I}{\scr{I}}
\newcommand{\J}{\scr{J}}
\newcommand{\C}{\scr{C}}
\newcommand{\W}{\scr{W}}
\newcommand{\F}{\scr{F}}
\newcommand{\lift}{\cancel{\square}}
\newcommand{\bo}[1]{\mathbf{#1}}
\newcommand{\spaces}{\bo{Spaces}}
\mathrmdef{Hom}

\def\defeq{\mathrel{\mathop:}=}
\def\presuper#1#2%
  {\mathop{}%
   \mathopen{\vphantom{#2}}^{#1}%
   \kern-\scriptspace%
   #2}
\begin{document}

\title{Right Delocalization of Model Categories}
\author{Bruce R Corrigan-Salter}
\eaddress{brcs@wayne.edu}
\address{Department of Mathematics\\ Wayne State University\\
1150 Faculty/Administration Building\\
656 W. Kirby 
\\Detroit, MI~48202, USA}

\keywords{localization, delocalization, model categories, diagram} 

\amsclass{18D99, 18E35, 55P60, 18G55}
\maketitle
\begin{abstract}
Model categories have long been a useful tool in homotopy theory, allowing many generalizations of results in topological spaces to other categories.  Giving a localization of a model category provides an additional model category structure on the same base category, which alters what objects are being considered equivalent by increasing the class of weak equivalences.  In some situations, a model category where the class of weak equivalences is restricted from the original one could be more desirable.  In this situation we need the notion of a delocalization.  In this paper, right Bousfield delocalization is defined, we provide examples of right Bousfield delocalization as well as an existence theorem.  In particular, we show that given two model category structures $\MO$ and $\MT$ we can define an additional model category structure $\MO \cap \MT$ by defining the class of weak equivalences to be the intersection of the $\MO$ and $\MT$ weak equivalences.  In addition we consider the model category on diagram categories over a base category (which is endowed with a model category structure) and show that delocalization is often preserved by the diagram model category structure.   
\end{abstract}

\section{\bf Introduction}
Since being introduced in \cite{quillen} by Quillen, model categories have long served as a means to work on homotopy theory.  Given a category, Quillen provides a list of necessary attributes on the category and three classes of maps, that allows one to give a decent description of the homotopy category obtained by formally inverting the class of weak equivalences.  This work was paramount in extending results found in topological spaces to other categories.

After their introduction, much has been done in the study of model categories and their associated homotopy categories.  One particular instance of work that has been done is the localization of model categories found in the work of Bousfield \cite{bousfieldI},\cite{bousfieldII}, \cite{bousfieldIII}, \cite{bousfieldIV} and \cite{bousfieldV} and presented in detail by Hirschhorn in \cite{hirschhorn} (see \ref{LOCDEF}).  By localizing a model category, the class of weak equivalences can be extended, allowing one to have a larger variety of "equivalent" objects. 

In this paper we consider the opposite of a localization, mainly a delocalization of a model category. By delocalizing we instead provide a model category with a smaller class of weak equivalences giving us the ability to expect extra structure on our class of weak equivalences.  In particular, we define the notion of right Bousfield delocalization and show that such a delocalization of a pair of model categories $\MO$ and $\MT$ exists when certain assumptions on $\MO$ and $\MT$ can be made.    

\subsection{Original Motivation}
In \cite{corrigansalterII} the author defines higher order Hochschild cohomology of a $k$-algebra $A,$ over a simplicial set $X_{\bullet}$ with coefficients in $M$ (denoted $H_{X_{\bullet}}^{\ast}(A,M)$).  The main result in this paper is determining what actions coefficient modules can take.  One question that can be asked about this higher order Hochschild cohomology is whether the cohomology is determined up to homotopy equivalence.  The problem is that in general, given two simplicial sets, the allowable coefficient modules may not be able to have the same amounts of actions (for example one simplicial set may only be able to work with symmetric bi-modules, while the other could work with bi-modules).  

One way to try and attack this problem is by considering an altered model category structure on the category of simplicial sets where weak equivalences are maps which are both weak equivalences and which preserve the "allowable modules" between the simplicial sets.  From a localization standpoint, this idea is somewhat backwards.  When localizing a model category, we get a larger class of weak equivalences, but here we are trying to restrict our class of weak equivalences.  By considering a "delocalization" we may find what we are looking for.  

\subsection{Organization}
We start in Section \ref{MODEL} by providing formal definitions for model categories, Bousfield localization and weak factorization systems.  We also use this section to set up the notation used in the paper.  In Section \ref{BOUS} we define and give examples of a Bousfield delocalization as well as an existence theorem.  In Section \ref{PROOF} the main delocalization theorem is proved and in Section \ref{DIAGRAM} we consider diagram categories and their delocalizations.  Lastly, we revisit the original motivation by looking back at the homotopy invariance question for higher order Hochschild cohomology in Section \ref{HOCH}.     

\section{\bf Model Categories}
\label{MODEL}
In this paper we assume some familiarity with model categories and their localizations, but we use this section to provide a formal definition for both along with the notation used within this paper.  Much of what is provided here will be used throughout the proofs later on.  To define a model category we make use of the following closely related concept introduced by Joyal and Tierney in \cite{joytier} and studied by Riehl in \cite{riehl}.
\begin{definition}\cite[2]{riehl}
  A \textbf{\textit{weak factorization system}}  $(L,R)$ on a category $\M$ consists of two classes of maps so that 
\begin{itemize}
\item Any map $f\in \M$ can be factored as $f=r\circ l$ with $l \in L$ and $r\in R.$
\item Any \textbf{\textit{lifting problem}}, i.e., any commutative diagram of the form

\begin{equation*}
\begin{tikzpicture}[baseline=(current bounding box.center)]
\matrix (m) 
[matrix of math nodes, row sep=3em, column sep=3em, text height=1.5ex, text depth=0.25ex]
{
\cdot & \cdot
\\
\cdot & \cdot \\ 
};
\path[->, thick, font=\scriptsize]
(m-1-1)
edge  (m-1-2)
edge node[anchor=east] {$l\in L$} (m-2-1)
(m-1-2)
edge node[anchor=west] {$r \in R$} (m-2-2)
(m-2-1) 
edge  (m-2-2);
\path[dashed,->] (m-2-1) edge (m-1-2)
; 
\end{tikzpicture}
\end{equation*}

has a \textbf{\textit{solution}}, i.e., a diagonal arrow making both triangles commute. 
\item The classes $L$ and $R$ are closed under retracts.
\end{itemize}
\end{definition}

It can be seen that by the retract argument $L$ and $R$ actually determine each other; $L$ is the class of maps with the left lifting property with respect to $R$ and $R$ is the class of maps with the right lifting property with respect to $L.$  For this reason, we will commonly denote classes of maps that lift in this manner as $R=L^{\lift}$ and $L=\presuper{\lift}{R}.$

With this, we give the definition of a model category.

\begin{definition}
A \textbf{\textit{model category structure}} on a category $\M$ consists of three classes of maps $\F, \C$ and $\W$ called fibrations, cofibrations and weak equivalences respectively so that $\W$ satisfies the 2 out of 3 property, and $(\C \cap \W, \F)$ and $(\C, \F \cap \W)$ are weak factorization systems.
\end{definition}

\subsection{Cofibrantly Generated Model Categories}
Many times a model category structure can be completely determined by a set of cofibrations and set of acyclic cofibrations (where the class of acyclic cofibrations is $\C \cap \W$).  When this happens we say that the model category is cofibrantly generated.  

\begin{definition}
\cite[2.1.17]{hovey}
Suppose $\M$ is a model category.  We say that $\M$ is \textbf{\textit{cofibrantly generated}} if there are sets $\I$ AND $\J$ of maps such that:
\begin{itemize}
\item[i)] the domains of the maps in $\I$ are small relative to $\I$-cell (transfinite compositions of pushouts of elements in $\I$)
\item[ii)] the domains of the maps in $\J$ are small relative to $\J$-cell (transfinite compositions of pushouts of elements in $\J$)
\item[iii)] the class of fibrations is $\J^{\lift}$ 
\item[iv)] the class of acyclic fibrations is $\I^{\lift}$ 
\end{itemize}
The sets of maps $\I$ and $\J$ are referred to as the sets of \textbf{\textit{generating cofibrations}} and \textbf{\textit{generating acyclic cofibrations}} respectively.
\end{definition}

\begin{remark}
In general, cofibrantly generated model categories only need the domains of $\I$ and $\J$ to be small relative to $\I$-cell and $\J$-cell respectively, however for this paper we typically assume a stronger notion of small, by expecting that the domains of $\I$ and $\J$ are small relative to the category $\M$ (see \cite[7.14]{dwyerspalinski}).
\end{remark}

Detecting whether a model category is cofibrantly generated can be done via the following theorem due to D.M. Kan.

\begin{theorem}
\cite[2.1.19]{hovey}
\label{HOVEY}
Suppose $\M$ is a category with all small colimits and limits. Suppose $\W$ is a class of maps of $\M,$ and $\I$ and $\J$ are sets of maps of $\M.$  Then there is a cofibrantly generated model category structure on $\M$ with $\I$ as the set of generating cofibrations, $\J$ as the set of generating acyclic cofibrations, and $\W$ as the class of weak equivalences if and only if the following conditions are satisfied.
\begin{itemize}
\item[i)] $\W$ has the 2 out of 3 property and is closed under retracts
\item[ii)] the domains of $\I$ are small relative to $\I$-cell
\item[iii)] the domains of $\J$ are small relative to $\J$-cell
\item[iv)] $\J$-cell$\subseteq \W \cap \presuper{\lift}{(\I^{\lift})}$ 
\item[v)] $\I^{\lift}\subseteq \W \cap \J^{\lift}$
\item[vi)] either $\W \cap \presuper{\lift}{(\I^{\lift})} \subseteq \presuper{\lift}{(\J^{\lift})}$ or $\W\cap \J^{\lift}\subseteq\I^{\lift}$.
\end{itemize}
\end{theorem}

\subsection{Bousfield Localization}
In this paper we use the notion of Bousfield localization extensively, so the definition will be provided here.  

\begin{definition}
\cite[3.3.1]{hirschhorn}
\label{LOCDEF}
Let $\M$ be a model category and $S$ be a class of maps in $\M.$
\begin{itemize}
\item[(1)] The \textbf{\textit{left Bousfield localization}} of $\M$ with respect to $S$ (if it exists) is a model category structure $L_S\M$ on the underlying category $\M$ such that
\begin{itemize}
\item[a)] the class of weak equivalences of $L_S\M$ equals the class of $S$-local equivalences of $\M$
\item[b)] the class of cofibrations of $L_S\M$ equals the class of cofibrations of $\M$ and
\item[c)] the class of fibrations of $L_S\M$ is the class of maps with the right lifting property with respect to maps which are both $S$-local equivalences and cofibrations.
\end{itemize} 
\item[(2)] The \textbf{\textit{right Bousfield localization}} of $\M$ with respect to $S$ (if it exists) is a model category structure $R_S\M$ on the underlying category $\M$ such that
\begin{itemize}
\item[a)] the class of weak equivalences of $R_S\M$ equals the class of $S$-colocal equivalences of $\M$
\item[b)] the class of fibrations of $R_S\M$ equals the class of fibrations of $\M$ and
\item[c)] the class of cofibrations of $R_S\M$ is the class of maps with the left lifting property with respect to maps which are both $S$-colocal equivalences and fibrations.
\end{itemize}
\end{itemize}
\end{definition}

\begin{remark} An alternate definition for Bousfield localizations is given in Section \ref{BOUS} so we refer the reader to \cite{hirschhorn} for more information on the definition provided above.
\end{remark}

\section{\bf Right Bousfield Delocalization}
\label{BOUS}
In this section we define right Bousfield delocalization, provide an existence theorem, as well as an example. First we start with an alternative definition of right Bousfield
localization.

\begin{definition} 
Given a category $\M$ with a model category structure which we denote
$\MO$ a right Bousfield localization of $\M$ is an additional model category structure on $\M$
(denoted $\MT$) with the following properties:
\begin{itemize}
\item if $f$ is a weak equivalence in the $\MO$ model category structure, then $f$ is a weak equivalence in the $\MT$ model category structure
\item $\MO$ and $\MT$ have the same class of fibrations.
\end{itemize}
\end{definition}

\begin{remark}
This definition agrees with \ref{LOCDEF} where $\MT$ is the right Bousfield localization of $\MO$ with respect to the class of $\MT$ weak equivalences.
\end{remark}

With this, the definition of a right Bousfield delocalization is quite natural.

\begin{definition}
Given a category $\M$ with a model category structure, which we denote $\MT,$ a \emph{right Bousfield delocalization} of $\MT$ is an additional model category structure on $\M$ (denoted $\MO$) with the property that $\MT$ is a right Bousfield localization of $\MO.$
\end{definition}

Given these definitions, recognizing a localization/delocalization relationship between
two model category structures on the same category $\M$ is fairly straight forward, but
providing the existence is more difficult. When starting with a model category structure
$\MO$ and considering the definition for Bousfield localizations provided in Section \ref{MODEL} (see \ref{LOCDEF}) we
could pick a class of maps with which we could try to localize $\MO$ with. In fact, under
certain assumptions on $\M$ and the associated class of maps, the localization is shown to exist.

\begin{theorem}\cite[5.1.1]{hirschhorn}
Let $\MO$ be a right proper cellular model category, $K$ be a set of objects in $\M,$ and $S$ the $K$-local equivalences, then the right Bousfield localization of $\MO$ with respect to $S$ exists.
\end{theorem}

\begin{remark}
In this paper we work primarily with right Bousfield delocalizations, but it is worth mentioning that there is also a well known analogous theorem on the existence of left Bousfield localizations \cite[4.1.1]{hirschhorn}
We would like a similar existence theorem for right Bousfield delocalizations and in order to provide one we must first consider the following definition.
\end{remark}

\begin{definition}
\label{INT}
Let $\M$ be a category with two model category structures $\MO$ and $\MT$ with cofibrations, fibrations and weak equivalences given be $\CO , \CT , \FO , \FT , \WO ,$ and $\WT$ respectively with the property that $\FO =\FT.$  We define the \textbf{\textit{right intersected}} model category of $\MO$ and $\MT$ (denoted $\MO \cap \MT$) to be the model category structure on $\M$ (if it exists) with the following classes of maps:
\begin{itemize}
\item fibrations $ \defeq \FO = \FT$
\item weak equivalences $\defeq \WO \cap \WT$
\item cofibrations $\defeq \presuper{\lift}{(\FO \cap \WO \cap \WT )}$
\end{itemize}
\end{definition}

\begin{remark}
It should be noticed that if $\MO \cap \MT$ exists, then it is a right Bousfield delocalization of both $\MO$ and $\MT.$
\end{remark}

With this, we can now provide an existence theorem for $\MO \cap \MT.$ 

\begin{theorem}
\label{MAIN}
Let $\M$ be a category with two cofibrantly generated model category structures on $\MO$ and $\MT$ with generating cofibrations $\IO$ and $\IT$ and generating acyclic cofibrations $\JO$ and $\JT$ respectively (whose domains are small relative to $\M$).  Assume also that fibrations $\FO$ and $\FT$ of $\MO$ and $\MT$ agree, then the model category structure of $\M,$ given by definition \ref{INT} exists.
\end{theorem}

\subsection{$A$-cellular model categories}

Before providing the proof (which is given in Section \ref{PROOF}) we consider an interesting
example.
As a homotopy theorist, when studying topological spaces, concern is placed on
$CW$ complexes up to homotopy equivalence. One way to make such considerations is
by using the Quillen model category structure on the category of topological spaces.
While $CW$ complexes are very useful, one could also consider the spaces obtained by building
with something besides the zero sphere (and the associated homotopy colimits). This
generalization to other ``$A$-cellular" spaces was studied by Dror Farjoun and Nofech in
\cite{dror} and \cite{nofech} respectively, where ``$A$-cellular" spaces are the cofibrant objects in the model category
structure on pointed topological spaces associated to a pointed cofibrant object $A$ (see
definition \ref{ACELL}) much like pointed $CW$ complexes are the cofibrant objects in the standard model
category structure on pointed topological spaces.
It turns out that this model category structure associated to $A$ is actually a right
Bousfield localization of the standard model category structure on pointed topological spaces. Nofech
defines this localized model category structure in more generality and starts with any
pointed closed simplicial model category (not necessarily pointed topological spaces).

\begin{definition}
\cite[1.0]{nofech}
\label{ACELL}
Let $C$ be a pointed closed simplicial model category and $A$ be a cofibrant object of $C.$  We call $C^A$ a closed model category structure associated with the localization of $C$ with respect to $A$ if the cofibrations, fibrations and weak equivalences in $C^A$ are defined as follows:
\begin{itemize}
\item $W_{C^A} \defeq \{\varphi \colon B \rightarrow C | \varphi_{f_\ast} \colon \Hom(A,\tilde{B})\rightarrow \Hom(A,C)$ is a weak equivalence of simplicial sets $\}$
\item $Fib_{C^A} \defeq Fib_C$
\item $Cof_{C^A} \defeq \presuper{\lift}{(W_{C^A} \cap Fib_{C^A})}$
\end{itemize}
where $\varphi_{f_\ast}$ is the fibrant approximation of $\varphi.$
\end{definition}

One question that can be asked in this right delocalization context is ``What model
category structure is obtained from the right intersection of two of these $A$-cellular model
categories?" For this we will consider the case when $C$ is the category of pointed topological
spaces $Top_\ast.$ We have the following proposition.

\begin{proposition}
Given two cofibrant pointed topological spaces $A$ and $A',$ $Top_\ast^A \cap Top_\ast^{A'} = Top_\ast^{A \vee A'}.$
\end{proposition}

\begin{proof}
For this proof we simply need to show that the model category $Top_\ast^{A\vee A'}$ can be described via the model category $Top_\ast^A \cap Top_\ast^{A'},$ but this amounts to showing that the weak equivalences agree.

Given a map $\varphi\colon B \rightarrow C$ we would like to show that both $\varphi_{f_\ast}^A \colon \Hom(A,\tilde{B})\rightarrow \Hom(A,C)$ and $\varphi_{f_\ast}^{A'} \colon \Hom(A',\tilde{B})\rightarrow \Hom(A',C)$ are weak equivalences if and only if $\varphi_{f_\ast}^{A\vee A'}\colon \Hom(A\vee A',\tilde{B})\rightarrow \Hom(A\vee A',C)$ is a weak equivalence.  We see that since $\Hom(A\vee A',-)\cong \Hom(A,-)\times\Hom(A',-),$ $\varphi_{f_\ast}^A$ and $\varphi_{f_\ast}^{A'}$ are weak equivalences, then $\varphi_{f_\ast}^{A\vee A'}$ must also be a weak equivalence.  Secondly, we have that $A$ and $A'$ are both retracts of $A \vee A'$ giving us that $\varphi_{f_\ast}^A$ and $\varphi_{f_\ast}^{A'}$ are both retracts of $\varphi_{f_\ast}^{A\vee A'}.$  This implies that if $\varphi_{f_\ast}^{A\vee A'}$ is a weak equivalence, then $\varphi_{f_\ast}^A$ and $\varphi_{f_\ast}^{A'}$ must also be weak equivalences.
\end{proof}

\subsection{Bousfield Quivers}
In \cite{salch} Salch describes the bi-colored quiver obtained by considering the existing model
category structures on a category $\M$ by letting the nodes be distinct model categories
and the directed edges be left and right Bousfield localizations (colored based on
whether the localization is right or left). An interesting fact that Salch finds is that
these Bousfield quivers are not in general connected.
A question that can be asked in this setting is ``Given two cofibrantly generated model
categories $\MO$ and $\MT$ whose fibrations agree, must $\MO$ and $\MT$ be contained in the same
component?" The answer is what might be expected.

\begin{corollary}
Given two cofibrantly generated model category structures on a category $\M$ denoted $\MO$ and $\MT$ (where generating cofibrations and generating acyclic cofibrations have small domains relative to $\M$); if fibrations agree, then $\MO$ and $\MT$ are in the same component of the Bousfield quiver discussed in \cite{salch} via the directed edges from $\MO \cap \MT.$
\end{corollary}

We omit the proof as this corollary is a direct consequence of Theorem \ref{MAIN}.

\section{\bf Proof of Theorem \ref{MAIN}}
\label{PROOF}

In the proof of Theorem \ref{MAIN} we not only show that the given cofibrantly generated model categories $\MO$ and $\MT,$ that $\MO \cap \MT$ is a model category structure on $\M,$ but that $\MO \cap \MT$ is cofibrantly generated as well.

\begin{proof}[of Theorem \ref{MAIN}]

Let $\M$ be a category with two cofibrantly generated model category structures $\MO$ and $\MT$ which agree on fibrations.  Let the generating cofibrations be $\IO$ and $\IT$ respectively and generating acyclic cofibrations be $\JO$ and $\JT$ respectively (all with small domains relative to $\M$) and weak equivalences be given by $\WO$ and $\WT$ respectively.  We aim to prove that $\MO \cap \MT$ is a cofibrantly generated model category by showing that the assumptions of Theorem \ref{HOVEY} are satisfied where $\I=\IO \cup \IT, \W=\WO \cap \WT, \J=\JO.$  For part $i)$ notice $\WO$ and $\WT$ both have the 2 out of 3 property and are closed under retracts, so certainly $\W=\WO \cap \WT$ is as well.  For parts $ii)$ and $iii)$ we assume that $\IO, \IT, \JO$ all have small domains relative to $\M.$  For part $iv)$ notice that 
$$\W \cap \presuper{\lift}{(\I^{\lift})}$$
$$=(\WO \cap \WT) \cap  \presuper{\lift}{((\IO \cup \IT)^{\lift})} \supseteq (\WO \cap \WT)\cap (\presuper{\lift}{(\IO^{\lift})} \cup \presuper{\lift}{(\IT^{\lift})})$$
$$=\WO \cap (\presuper{\lift}{(\IO^{\lift})} \cup \presuper{\lift}{(\IT^{\lift})})\cap \WT \cap (\presuper{\lift}{(\IO^{\lift})} \cup \presuper{\lift}{(\IT^{\lift})})$$
but $\JO-cell \subseteq \WO \cap \presuper{\lift}{(\IO^{\lift})}$ and since $\MT$ could have been cofibrantly generated by $\IT$ and $\JO$ as cofibrations and acyclic cofibrations respectively (since fibrations between $\MO$ and $\MT$ agree) we also see that $\JO-cell \subseteq \WT \cap \presuper{\lift}{(\IT^{\lift})}$ which gives us that $\J-cell \subseteq \W \cap \presuper{\lift}{(\I^{\lift})}.$

For part $v)$ we see that $\J^{\lift} = \JO^{\lift} = \JT^{\lift}$ since fibrations agree so $\W\cap \J^{\lift} = (\WO \cap \JO^{\lift})\cap (\WT \cap \JT^{\lift}).$  We also have that $\IO^{\lift} \subseteq \WO \cap \JO^{\lift}$ and $\IT^{\lift} \subseteq \WT \cap \JT^{\lift},$ but since $\I^{\lift} =\IO^{\lift} \cap \IT^{\lift}$ this implies that $\I^{\lift} \subseteq \W\cap \J^{\lift}.$  Lastly, for part $vi)$ notice that $\W\cap \J^{\lift} = (\WO \cap \J^{\lift})\cap (\WT \cap \J^{\lift})$ and $\I^{\lift}=\IO^{\lift}\cap \IT^{\lift},$ but$(\WO \cap \J^{\lift})\subseteq \IO^{\lift}$ and $(\WT \cap \J^{\lift})\subseteq \IT^{\lift}$ since $\J^{\lift} = \JO^{\lift} = \JT^{\lift}$ so $\W \cap \J^{\lift} \subseteq \I^{\lift}$ and therefore $\MO \cap \MT$ is cofibrantly generated by $\IO \cup \IT$ and $\J.$

\end{proof}

\section{\bf Delocalization of diagram categories}
\label{DIAGRAM}

A good reason for localizing a model category $\M$ is to change the cofibrant-fibrant objects while maintaining much of the data from the original model category structure on $\M.$  The purpose of doing so is  so that the homotopy category associated to the localization (determined by its cofibrant-fibrant objects) can be modified accordingly.  One example of this is the creation of a homotopy category which describes homotopy algebras over algebraic theories, semi-theories, multi-sorted algebraic theories and finite product sketches studied by Badzioch in \cite{badziochI} and \cite{badziochII}, Bergner in \cite{bergner} and the author in \cite{corrigansalterI} respectively.

\begin{proposition}\cite[4.5]{corrigansalterI}
\label{FINITE}
Let $(C,\kappa)$ be a finite product sketch i.e. 
$C$ is a small category and $\kappa$ is a set of cones in $C.$  The full subcategory of $\spaces^C$ spanned by the homotopy $(C,\kappa)$-algebras with objectwise weak equivalences inverted is equivalent to the homotopy category of $L\spaces^C,$ where $L\spaces^C$ is a suitable left Bousfield localization of $\spaces^C.$
\end{proposition}

Before going much further we will take a look at the category $\spaces^C,$ or more generally $\M^C$ and the associated model category structure (implied to exist in Proposition \ref{FINITE}).

\begin{definition}
Let $\M$ be a category and $C$ be a small category.  By $\M^C$ we mean the \emph{category of diagrams in $\M$ with the shape of $C.$}  Objects in $\M^C$ are functors from $C$ to $\M$ and morphisms are natural transformations.
\end{definition}

If $\M$ is a cofibrantly generated model category, then we have an induced model category structure on $\M^C.$

\begin{theorem}
\cite[11.6.1]{hirschhorn}
\label{OBJFIB}
If $C$ is a small category and $\M$ is a cofibrantly generated model category with generating cofibrations $\I$ and generating acyclic cofibrations $\J,$ then $\M^C$ is a cofibrantly generated model category with the following classes of maps:
\begin{itemize}
\item weak equivalences are objectwise weak equivalences
\item fibrations are objectwise fibrations
\item cofibrations are maps with the left lifting property with respect to objectwise acyclic fibrations.
\end{itemize}
\end{theorem}

Since fibrations are defined objectwise in $\M^C$ from Theorem \ref{OBJFIB}, we will refer to this model category structure as the \textbf{\textit{objectwise fibration}} model category structure.  This theorem gives the existence of the objectwise fibration model category structure if the model category $\M$ is cofibrantly generated, but in general, if these classes of maps in $\M^C$ satisfy the axioms for a model category, then we will refer the the resulting model category as the objectwise fibration model category as well.

With this, we have a proposition relating to the intersection delocalization.

\begin{proposition}
\label{DIAGAG}
Let $\M$ be a category with two cofibrantly generated model category structures $\MO$ and $\MT$ where fibrations agree and $\MO \cap \MT$ gives a cofibrantly generated model category structure for $\M,$ then the model category structures $(\MO \cap \MT)^C$ and $\MO^C \cap \MT^C$ agree.
\end{proposition}

The proof of Proposition \ref{DIAGAG} follows directly from definition so we leave it to the reader.

When working with diagram categories and their model category structures, one question that can be asked is ``Given a model category $\M,$ and small category $C$ where the objectwise fibration model category $\M^C$ exists; for a right (left) Bousfield localization of $\M^C,$ is there a model category structure $\MT$ on $\M$ so that the objectwise fibration model category structure $\MT^C$ agrees with the Bousfield localization of $\M^C?$"  An answer in the affirmative would certainly make dealing with localized diagram model category structures nicer, however in general the answer for Bousfield localizations is no.  We do show though, that for the right intersected delocalization model categories, the analogous questions is easier to satisfy.  We do so by providing a detection theorem for an induced model category structure on $\M$ from a model category structure on $\M^C.$

Before providing the detection theorem and subsequent delocalization theorem, we will define a few needed notions.

\begin{definition}
Let $\M$ be a category and $C$ be a small category.  We have a functor:
$$P\colon \M \rightarrow \M^C$$
where $P(A)\colon C \rightarrow \M$ evaluates as $A$ on every object of $C.$  We refer to this functor as the \textbf{\textit{pointed functor}}, objects $P(A)\in \M^C$ as \textbf{\textit{pointed objects}} and maps $P(f)$ as \textbf{\textit{pointed maps}}.
\end{definition}

\begin{definition}
For a map $f \in \M^C,$ we will refer to the maps $f(\alpha)\colon X(\alpha)\rightarrow Y(\alpha)$ where $\alpha \in C$ as \textbf{\textit{component maps}} (specifically the component map of $f$ at $\alpha$).
\end{definition}

With this we have the following theorem.

\begin{theorem}
\label{DIAGDOWN}
Let $C$ be a small category, $\M$ be a category closed under finite limits and colimits.  If $\M^C$ has a model category structure so that the following is true:
\begin{itemize}
\item[i)] Given a map $f\colon X \rightarrow Y$ in $\M^C,$ if for all $\alpha \in C$ the component map of $f$ at $\alpha$ $f(\alpha)\colon X(\alpha)\rightarrow Y(\alpha) \in \M$ is a component map for some fibration $g\colon X' \rightarrow Y'$ of $\M^C,$ then $f$ is a fibration
\item[ii)]Given a map $f\colon X \rightarrow Y$ in $\M^C,$ if for all $\alpha \in C$ the component map of $f$ at $\alpha$ $f(\alpha)\colon X(\alpha)\rightarrow Y(\alpha) \in \M$ is a component map for some weak equivalence $g\colon X' \rightarrow Y'$ of $\M^C,$ then $f$ is a weak equivalence
\item[iii)] If $f\colon X_\ast \rightarrow Y_\ast$ is a pointed map in $\M^C,$ which has the left lifting property with respect to all pointed fibrations, then $f$ is a weak equivalence.
\item[iv)] The sub-classes of fibrations and acyclic fibrations in the full pointed subcategory of $\M^C,$denoted $\F_\ast$ and $\F_\ast\cap \W_\ast$ respectively, have the property that (in the full pointed subcategory) $(\presuper{\lift}{(\F_\ast)})^{\lift}=\F_\ast$ and $(\presuper{\lift}{(\F_\ast\cap \W_\ast)})^{\lift}=\F_\ast\cap \W_\ast$ 
\item[v)] If a pointed map $f$ lifts with respect to all pointed acyclic fibrations and $f$ is a weak equivalence, then $f$ lifts with respect to all pointed fibrations
\end{itemize}
then $\M$ has a model category structure with the following classes of maps:
\begin{itemize}
\item fibrations $\defeq$ the class of all possible component maps of fibrations in $\M^C$
\item weak equivalences $\defeq$ the class of all possible component maps of weak equivalences in $\M^C$
\item cofibrations $\defeq$ maps with the left lifting property with respect to maps which are both fibrations and weak equivalences.
\end{itemize}
Furthermore, the objectwise fibration model category structure on $\M^C$ using the fibrations and weak equivalences just defined agrees with the original model category structure on $\M^C$ (assumed to exist in the assumptions of this theorem).
\end{theorem}

Before getting to the proof we will consider the following lemma.

\begin{lemma}
\label{LEM}
Let $(L,R)$ is a weak factorization system for $\M^C$ with the property that for any map $f\in \M^C,$ if for all $\alpha \in C$ the component maps of $f$ at $\alpha$ are each component maps for some map $g\in R$ then $f\in R.$  Then for $f \in L$ $$f(\alpha)\colon X(\alpha)\rightarrow Y(\alpha)$$ belongs to $\presuper{\lift}{R_\ast}$ where $R_\ast$ is the class of maps in $\M$ which show up as component maps for fibrations in $\M^C.$
\end{lemma}

\begin{proof}
We start with an adjoint pair of functors for $\alpha \in C$:
$$F\colon \M^C \rightleftarrows \M \colon G$$
where $F(X)=X(\alpha)$ for $X\in \M^C$ and $G(y)(\beta)=\prod\limits_{\beta \rightarrow \alpha}y$ for $y\in \M$ and $\beta \in C.$  Suppose now that $f\in L$ and there is a commutative diagram:

\begin{equation*}
\begin{tikzpicture}[baseline=(current bounding box.center)]
\matrix (m) 
[matrix of math nodes, row sep=3em, column sep=3em, text height=1.5ex, text depth=0.25ex]
{
X(\alpha) & w \\
Y(\alpha) & z \\ 
};
\path[->, thick, font=\scriptsize]
(m-1-1)
edge  (m-1-2)
edge node[anchor=east] {$f(\alpha)$} (m-2-1)
(m-1-2)
edge node[anchor=west] {$g$} (m-2-2)
(m-2-1) 
edge  (m-2-2)
; 
\end{tikzpicture}
\end{equation*}

then by the adjoint pair $(F,G)$ we get a diagram;

\begin{equation*}
\begin{tikzpicture}[baseline=(current bounding box.center)]
\matrix (m) 
[matrix of math nodes, row sep=3em, column sep=3em, text height=1.5ex, text depth=0.25ex]
{
X & G(w)
\\
Y & G(z) \\ 
};
\path[->, thick, font=\scriptsize]
(m-1-1)
edge  (m-1-2)
edge node[anchor=east] {$f$} (m-2-1)
(m-1-2)
edge node[anchor=west] {$G(g)$} (m-2-2)
(m-2-1) 
edge  (m-2-2)
; 
\end{tikzpicture}
\end{equation*}

Now if we suppose $g \in R_\ast$ then $G(g)\in R$ since $R$ is closed under arbitrary products, so in particular, there exists a map $\theta \colon Y \rightarrow G(w)$ that gives a lift in the second diagram above and therefore we have a map (using the adjoint pair of functors) $\varepsilon_w \circ F(\theta) \colon Y(\alpha)\rightarrow w$ so that a lift exists in the first diagram and $f(\alpha)\in \presuper{\lift}{R_\ast}.$

\end{proof}

We can now prove Theorem \ref{DIAGDOWN}.

\begin{proof}[of Theorem \ref{DIAGDOWN}]
Let $\C, \F$ and $\W$ be the cofibrations, fibrations and weak equivalences described in Theorem \ref{DIAGDOWN}.  We start by proving that $\M$ has a model category structure by showing that the pairs $(\C, \F \cap \W)$ and $(\C\cap \W, \F)$ are weak factorization systems and that $\W$ satisfies the 2 out of 3 property.  For the 2 out of 3 property, consider the composition of two maps $f,g \in \M.$  We have the pointed functor $P\colon \M \rightarrow \M^C.$  Note that $F(f), F(g)$ and $F(fg)$ are weak equivalences if and only if $f,g$ and $fg$ are as well.  From this, it is clear that $\W$ satisfies the 2 out of 3 property.  

For the weak factorization systems, notice that $(\C, \F\cap \W)$ is a weak factorization system since $\C$ is defined to be the maps that have the left lifting property with respect to $\F\cap \W$ and any map $f \in \M$ has an associated pointed map $f \in \M^C$ which can be factored as $p\circ i$ where $i$ is a cofibration and $p$ is an acyclic fibration, but by lemma \ref{LEM}, any component map of $i$ lifts with respect to pointed fibrations i.e. component maps of $i$ are in $\C$ and by definition any component map of $p$  belongs to $\F \cap \W.$  Now, for $(\C\cap \W, \F)$ any diagram of the form 

\begin{equation*}
\begin{tikzpicture}[baseline=(current bounding box.center)]
\matrix (m) 
[matrix of math nodes, row sep=3em, column sep=3em, text height=1.5ex, text depth=0.25ex]
{
 A & B 
\\
C & D \\ 
};
\path[->, thick, font=\scriptsize]
(m-1-1)
edge  (m-1-2)
edge node[anchor=east] {$\psi$} (m-2-1)
(m-1-2)
edge node[anchor=west] {$\varphi$} (m-2-2)
(m-2-1) 
edge  (m-2-2)
; 
\end{tikzpicture}
\end{equation*}
(where $\psi\in \C \cap \W$ and $\varphi \in \F$)
has a lift since by part $v)$ any pointed map  
$f$ of $\M^C$ which lifts with respect to all pointed acyclic fibrations which are also weak equivalences, lifts with respect to all pointed fibrations, but when considering the diagram above as a pointed diagram in $\M^C$ we see that a lift does exist which gives a lift in $\M$ as well.  Also, any map $f \in M$ has a factorization since the associated pointed map in $\M^C$ has a factorization $P(f)=p\circ i$ where $i$ is an acyclic cofibration and $p$ is a fibration.  $i$ is a cofibration which gives that component maps are in $\C$ by Lemma \ref{LEM} and $i$ is a weak equivalence so component maps are also weak equivalences and in $\W.$  We also see that component maps of $p$ are in $\F,$ but this gives that any composition of component maps from $p$ and $i$ form a factorization for $f.$

To finish the proof we need to show that this model category structure on $\M$ gives the same model category structure on $\M^C$ when considering the objectwise fibration model category structure.  This is a straight forward check of the defined weak equivalences and fibrations, so we leave the remainder of the proof to the reader. 
\end{proof}

This brings us to the intersection delocalization theorem discussed in the beginning of the section.

\begin{theorem}
Let $\M$ be a co complete and complete category and $C$ be a small category.  Given two model category structures on $\M^C$ denoted $\M_1^C$ and $\M_2^C$ which both induce a model category structure on $\M$ (see Theorem \ref{DIAGDOWN}) where the right Bousfield delocalization model category $\M_1^C \cap \M_2^C$ exists and with the assumption that $\M_1^C \cap \M_2^C$ satisfies part $v)$ of Theorem \ref{DIAGDOWN} then there exists a model category structure $\M_3$ on $\M$ whose associated objectwise fibration model category structure on $\M^C$ agrees with $\M_1^C \cap \M_2^C$ and furthermore $\M_3$ and $\M_1 \cap \M_2$ agree.
\end{theorem}

The proof of this theorem is a straight forward check that $\M_1^C \cap \M_2^C$ satisfies the assumptions of Theorem \ref{DIAGDOWN}.  Once this is done, it is clear that $\M_3$ and $\M_1 \cap \M_2$ agree by the defined classes of maps.

\section{\bf Homotopy invariance for Hochschild cohomology}
\label{HOCH}

Although it was the original motivation for this work, in this section we show that a right intersected delocalization does not seem to fit when trying to determine whether higher order Hochschild cohomology is homotopy invariant.  Recall that we would like to consider the category of pointed simplicial sets where weak equivalences are maps which are traditional weak equivalences of pointed simplicial sets, but which also preserve the types of modules which can be used for each simplicial set.

By using a right Bousfield delocalization of the traditional model category structure on pointed simplicial sets, we see that the fibrations would remain the same, but this also implies that the acyclic cofibrations would remain the same.  This is problematic because one could easily come up with maps which would still be considered weak equivalences although the modules between the two simplicial sets differ.  For an example, consider the smallest simplicial set for $S^1$ (with one non-degenerate 0-simplex and one non-degenerate 1-simplex) as the first simplicial set.  For the second simplicial set consider the same simplicial decomposition of $S^1,$ but in addition we will add a 1-simplex whose 0th face is identified with the original 0-simplex.  There is a natural inclusion of pointed simplicial sets, which is both a weak equivalence of simplicial sets and a cofibration, but the coefficient modules do not agree.  

\section*{Acknowledgment}
I would like to thank Andrew Salch, Daniel Isaksen, Bernard Badzioch and Matthew Sartwell for many conversations about this work.  I would also like to thank my wife for her continued support.


\end{document}